\documentclass[10pt]{amsart}

\usepackage{amssymb,latexsym, amssymb, amsmath}
\usepackage{enumerate}
\usepackage{amsfonts}
\usepackage{color}
\usepackage{comment}

\usepackage{hyperref}

\newtheorem{thm}{Theorem}[section] 
\newtheorem{lem}[thm]{Lemma}     
\newtheorem{cor}[thm]{Corollary}
\newtheorem{prop}[thm]{Proposition}

\newtheorem{cond}{Conditions}

\numberwithin{equation}{section} 
\numberwithin{thm}{section}

\newtheorem{remark}[thm]{Remark}

\newcommand{\N}{\mathbb{N}}
\newcommand{\R}{\mathbb{R}}
\newcommand{\C}{\mathbb{C}}
\newcommand{\Z}{\mathbb{Z}}

\newcommand{\T}{\mathbb{T}}

\newcommand{\vektor}[1]{\mathbf{#1}}
\newcommand{\av}{\vektor{a}}
\newcommand{\bv}{\vektor{b}}

\newcommand{\nv}{\vektor{n}}
\newcommand{\Nv}{\vektor{N}}
\newcommand{\mv}{\vektor{m}}
\newcommand{\xv}{\vektor{x}}
\newcommand{\yv}{\vektor{y}}
\newcommand{\hv}{\vektor{h}}
\newcommand{\rv}{\vektor{r}}
\newcommand{\tv}{\vektor{t}}

\newcommand{\tri}{\chi}

\newcommand{\alv}{\boldsymbol{\alpha}}
\newcommand{\bev}{\boldsymbol{\beta}}
\newcommand{\gamv}{\boldsymbol{\gamma}}

\newcommand{\einsv}{\boldsymbol{1}}

\providecommand{\lcm}{\mathop{\rm lcm}\nolimits}

\newcommand{\Sing}{{\mathfrak S}}
\newcommand{\ma}{{\mathfrak m}}
\newcommand{\Ma}{{\mathfrak M}}

\title[A quadric in dense variables]
 {On a diagonal quadric in dense variables}

\author{Eugen Keil}

\begin{document}
\maketitle

\begin{abstract}
We examine the solubility of a diagonal, translation invariant, quadratic equation system
in arbitrary (dense) subsets $\mathcal{A} \subset \Z$ and show quantitative bounds on the size of $\mathcal{A}$
if there are no non-trivial solutions.
We use the circle method and Roth's density increment argument.
Due to a restriction theory approach we can deal with equations in $s \geq 7$ variables.
\end{abstract}

\section{Introduction} \label{introduction}


Diophantine equations with an underlying symmetry group appear naturally in many
number theoretic questions and their structure can often be exploited to obtain
solutions in situations where standard methods fail.
This is one of the main reasons why translation invariant systems have attracted so
much interest. 
The important special case of linear systems was approached by methods ranging from 
Fourier-analysis \cite{Roth} over ergodic theory \cite{Fur2} to (additive) 
combinatorics \cite{Gow}, \cite{Szem} and gave rise
to recent amazing developments in the theory of linear equations in the primes \cite{GrTao1}.

A first attempt to introduce non-linear terms was made by Lov\'asz. He conjectured that
\begin{align*}
x_1 - x_2 = y^2
\end{align*}
with $x_1,x_2 \in \mathcal{A} \subset \N$ and $y \in \N$ has solutions for any dense set $\mathcal{A}$.
This was proven independently by Furstenberg \cite{Fur2} and S\'ark\"ozy \cite{Sar}, and a simplified
proof can be found in \cite{Gr2}.
Generalisations of this result replace $y^2$ by more general polynomials
and consider more equations. An important result arising from this work is the polynomial Szemer\'edi theorem
of Bergelson and Leibman \cite{BerLeib}.

Quite remarkably, the case of higher degree in the variables $x_i$ was neglected until
Smith \cite{Smith} considered a family of translation invariant quadratic equation systems
\begin{equation} \label{eq-sys}
\begin{split} 
\lambda_1 x_1^2 + \lambda_2x_2^2 + \ldots + \lambda_s x_s^2 = 0, \\
\lambda_1 x_1 + \lambda_2x_2 + \ldots + \lambda_s x_s = 0,
\end{split}
\end{equation}
where $\lambda_i \in \Z$ and the $x_i$ are restricted to an arbitrary 
set $\mathcal{A} \subset \{1,2,\ldots, N\}$. 
For a translation invariant system we expect non-trivial solutions if
the set $\mathcal{A}$ has many elements. We ask the quantitative question: 
Assume there are only trivial solutions to \eqref{eq-sys}, what can we say about the size of $\mathcal{A}$?

Smith \cite{Smith} uses the circle method and uniformity norms to obtain an 
upper bound on the cardinality of $\mathcal{A}$ of size 
$CN(\log\log N)^{-c}$ for some $C,c > 0$ as long as $s \geq 9$.
In other words, any set with larger cardinality has non-trivial solutions.
Our goal here is to improve on his work by removing the uniformity norms from the argument and 
reducing the number of variables down to $s \geq 7$.
This is done under the following natural conditions.
\begin{cond}\ \\
(i) $\lambda_1 + \lambda_2 + \ldots + \lambda_s = 0$,\\
(ii) $s \geq 7$ and $\lambda_i \neq 0$ for all $1 \leq i \leq s$,\\
(iii) there are at least two positive and at least two negative coefficients $\lambda_i$.
\end{cond}

The first condition ensures translation invariance of the system, which is the key property in
problems of this type, as explained above. 
If system \eqref{eq-sys} is not translation invariant, it is easy to construct dense
sets without non-trivial solutions by using divisibility obstructions.
Choose for instance $\mathcal{A}$ to be the set of numbers congruent to $1$ modulo $n$ for a 
natural number $n > |\lambda_1 + \lambda_2 + \ldots + \lambda_s|$.

Before we discuss the other two conditions and give a historical overview, we state the main theorem of this work.
Write $|\mathcal{A}|$ for the cardinality of $\mathcal{A}$ and call a solution with $x_i = x_j$ for some $i \neq j$
a \emph{trivial solution} of \eqref{eq-sys}.

\begin{thm} \label{dio-thm}
Assume that the conditions above hold and system \eqref{eq-sys} has only trivial solutions 
for $x_i \in \mathcal{A} \subset \{1,2,\ldots, N\}$.
Then we have the bound $|\mathcal{A}| \leq CN (\log\log N)^{-1/15}$ for some constant $C$, which depends 
on the coefficients $\lambda_i$ of the system.
\end{thm}

The bound $s \geq 7$ in condition $(ii)$ seems to be best possible in what can be achieved 
by this type of Fourier-analytic methods.
If we assume the existence of non-trivial integer solutions of \eqref{eq-sys}, then a qualitative
version of Theorem \ref{dio-thm} can be deduced from Szemeredi's theorem for any $s \geq 4$
as follows.
First we locate an arithmetic progression in $\mathcal{A}$ of length $2k+1$. By homogeneity and translation invariance of 
\eqref{eq-sys}, we can rescale to $\{-k, \ldots, k\}$ without changing the shape of the system. 
When $k$ is large enough, there are non-trivial solutions by assumption on the system and we are done.

To see why condition $(iii)$ is needed, we observe that
conditions $(i)$ and $(ii)$ imply that there is at least one positive and at least one negative coefficient in \eqref{eq-sys}.
But in situations where we have exactly one negative (positive) coefficient $\lambda_i$, we have only
trivial solutions. To see this, we use the translation invariance. If there is a non-trivial solution,
there has to be one with $x_i = 0$. The remaining quadratic equation is 
positive (negative) definite and this forces the other variables to be zero as well.
Therefore, condition $(iii)$ is necessary and we use it to ensure the existence of a non-singular real
solution to \eqref{eq-sys} in our application of the circle method.

It is possible to derive a simple corollary from Theorem \ref{dio-thm}.

\begin{cor}
Under the conditions above, a diagonal quadratic form
\begin{align*}
\lambda_1 x_1^2 + \lambda_2x_2^2 + \ldots + \lambda_s x_s^2 = 0
\end{align*}
has non-trivial solutions in any set $\mathcal{A} \subset \{1,2,\ldots, N\}$
of density $|\mathcal{A}|/N \geq C (\log\log N)^{-1/15}$ for a constant $C$, which depends
on the coefficients $\lambda_i$ of the system.
\end{cor}

It is apparent that better quantitative estimates in Theorem \ref{dio-thm} lead to a wider range of 
sets $\mathcal{A}$ which are covered by the result. One of the major (widely open) goals
is to improve on those bounds such that they cover interesting 
number theoretic sets, such as the primes. (One can handle the corollary in the case of prime numbers by classical techniques already for $s \geq 5$.) 
 
The general interest in quantitative estimates goes back to a conjecture of Erd\H{o}s and Turan \cite{ErTu} from 1936.
Does every subset 
\begin{align*}
\mathcal{A} = \{a_1 < a_2 < a_3~<~\ldots\}~\subset~\N
\end{align*}
contain an arithmetic progression of arbitrary length $k$, if all
we know is that 
\begin{align*}
\sum_{i = 1}^{\infty} \frac{1}{a_i} = \infty?
\end{align*}
The first progress was made by Roth \cite{Roth}, who showed that in the absence of
non-trivial arithmetic progressions of length three the size of $\mathcal{A}$
in an interval of length $N$ is bounded by $CN(\log\log N)^{-1}$ for some constant $C > 0$.
There have been several improvements on Roth's theorem, for example by
Heath-Brown \cite{HB}, who obtained an upper bound $|\mathcal{A}| \leq CN\log^{-c}N$
for some small $c > 0$ and later Bourgain \cite{Bour3}
with a bound of the form $|\mathcal{A}| \leq CN(\log N)^{-2/3}(\log\log N)^2$. 
A recent breakthrough is the result $|\mathcal{A}| \leq CN(\log N)^{-1}(\log\log N)^5$ of Sanders \cite{San}, 
which is remarkable since an estimate of the form $CN(\log N)^{-1}(\log\log N)^{-2}$
would be enough to settle the conjecture of Erd\H{o}s and Turan in the case $k = 3$.

In the case of general $k$-term progressions, where the classical Fourier-analytic approach fails,
the bounds are even weaker.
Szemeredi \cite{Szem} proved by a complicated combinatorial method that 
$|\mathcal{A}|/N \to 0$ as $N \to \infty$, but his method gave no useful bound on $|\mathcal{A}|$.
Gowers \cite{Gow} gave a quantitative bound of size $C_kN(\log\log N)^{-c_k}$ for arbitrary $k \geq 4$
by using `uniformity norms' and a variant of Freiman's theorem.
As mentioned earlier, there is an ergodic theoretic approach by
Furstenberg \cite{Fur2} (see \cite{Fur} for an exposition), who gave a simpler proof for Szemeredi's result 
without any explicit bound at all. 

Our method here is essentially the one used by Roth. There seem to be no obvious 
generalisation of the other methods to improve on the density estimate in our result, but it is reasonable
to conjecture that the $(\log\log N)^{-1/15}$ is still far away from the truth.

\emph{Acknowledgements:}
This paper is part of the author's Ph.D. thesis and he would like to express many thanks to 
his supervisor Trevor Wooley for his time, support and motivation.
The author also would like to thank the referee for comments that helped to improve the exposition.
The author's doctoral studies were partially supported by EPSRC.

\section{Notation and Remarks} \label{notation}

Many sections of this paper are independent of each other and it is convenient
to introduce some of the notation at the beginning of each section, where it is needed. 
At this point we focus on the most important things, which are relevant throughout the paper.

Write $e(x) = \exp(2\pi i x)$ and $e_q(x) = e(x/q)$ for $q \in \N$.
We use the usual notation $f = O(g)$ to express that $|f| \leq Cg$ for some constant $C > 0$ and similarly
Vinogradov's notation $f \ll g$. We indicate dependencies 
on parameters by subscripts as in $O_p(N^s)$ or $\ll_{P, \epsilon}$, for example.
We write $\widehat{f}$ for the Fourier transform of a function $f:\R \to \R$. 
Bold face letters such as $\xv$ denote vectors with components $x_i$ and $\xv \cdot \yv$ is the usual scalar product.
Inequalities such as $\xv \leq P$ or $\xv \leq \yv$ should be understood componentwise.
Write $(a;q)$ and $(\av;q) := (a_1;a_2;q)$ for the greatest common divisor of the components 
and $\lcm(q_1,\ldots,q_k)$ for the least common multiple of $q_1, \ldots, q_k$.
A sum over natural numbers starts at one, if not otherwise indicated and
we use the following abbreviations in our summations.
\begin{align*}
\sum_{(a;q)=1} := \sum_{\substack{a=1\\(a;q)=1}}^q
\mbox{\qquad and \qquad} \sum_{(\av;q)=1} := \sum_{a_1 = 1}^q\sum_{\substack{a_2=1\\(\av;q)=1}}^q
\end{align*}

The symbol $\T$ is used to refer to the `circle' $\R/\Z$ with the circle norm
$\|\alpha\| := \min\{|\alpha - z|: z \in \Z\}$,
the distance of $\alpha \in \R$ to the nearest integer.
Let $f: \T^d \to \C$ be a function on $\T^d$ and $g: B \to \C$ a function with support inside a unit box 
$B = \prod_{i = 1}^d [y_i,y_i+1] \subset \R^d$.
The expression 
\begin{align*}
\int_{\T^d} f(\alv) g(\alv) \,d\alv
\end{align*} 
should be understood as the integral over $B$. 
We take the notational freedom to identify $\T$ with different specific intervals
of length one, such as $[0,1]$ or $[-1/2,1/2]$, for example.
This is well defined due to the one-periodicity of $f$.

The parameter $N \in \N$, governing the size of the variables $x_i$ in \eqref{eq-sys}, is the most important
parameter in this paper and should be thought of as large.
The parameter $Q \in \N$ is used for the major-minor-arc decomposition and we assume
throughout the paper that $64 \, Q^{52} \leq N$. It will be set to a small power of $N$ eventually.

Now we define the main objects of our study, which are certain quadratic exponential sums
associated to the system \eqref{eq-sys}.
For a function $g: \N \to \C$ and $\mathcal{A} \subset \{1,2,\ldots,N\}$ we define
\begin{align} \label{eq-Sg}
V_g(\alv) = \sum_{n \leq N} g(n) e(\alpha_2 n^2 + \alpha_1 n) \mbox{\ \ and \ }
V_\mathcal{A}(\alv) = \sum_{n \in \mathcal{A}} e(\alpha_2 n^2 + \alpha_1 n)
\end{align}
and write $V(\alv)$ for the sum without the weight function $g$. The local versions
\begin{align} \label{eq-v}
v(\alv) = \int_{0}^N e(\alpha_2 t^2 + \alpha_1 t) \,dt, \qquad V(q,\av) = \sum_{r = 1}^q e_q(a_2 r^2 + a_1 r).
\end{align}
are needed in the argument as well as the linear exponential sum 
$L(\alv) := L_{3N}(\alpha_1)L_{3N^2}(\alpha_2)$, where $L_M$ is the usual one-dimensional version
\begin{align} \label{lin-exp}
L_M(\alpha) = \sum_{n \leq M} e(\alpha n).
\end{align}
As is common in the circle method approach, the sums $V_\mathcal{A}$ are used to 
encode the number of solutions to system \eqref{eq-sys} as an integral \eqref{eq-sys-sol}.
To evaluate it, we need to bound $L^p$-means of the exponential sums $V_g$ for arbitrary $|g| \leq 1$.
Due to the presence of arbitrary coefficients in \eqref{eq-Sg}, 
Weyl-differencing techniques are not available and Parseval's identity works only for $p \geq 8$.
This is the main reason why the result of Smith \cite{Smith} gives Theorem \ref{dio-thm} only for $s \geq 9$.
Here we use methods inspired by `restriction theory' instead to prove
the following result on $L^p$-estimates of these exponential sums, which
allows us to reduce the number of variables in our analysis down to $s \geq 7$.

\begin{thm} \label{SA-est}
Let $V_g(\alv)$ be defined as in \eqref{eq-Sg} for a function $g$ with $|g(n)| \leq 1$. 
Then for $p > 6$, we have
\begin{align*}
\int_{\T^2} |V_g(\alv)|^p \,d\alv \ll_p N^{p-3}.
\end{align*}
\end{thm}

Theorem \ref{SA-est} follows from a more general result of Bourgain \cite[Prop. 3.114]{Bour2}.
We give here a new proof, which differs from that of Bourgain in some technical and 
conceptual points and has strong parallels with the approach in \cite{Gr}, where
similar estimates played a major role in the proof of Roth's theorem in the primes.
In our case we deduce it from Theorem \ref{Lp-est}, which is a general theorem
for handling $L^p$-means of exponential sums.

While Theorem \ref{SA-est} is a corollary of a known result, 
the proof methods below also give variants of Theorem \ref{SA-est},
which are not covered by Bourgain's work. Similarly, Theorem \ref{Lp-est} seems to be a 
new result and might prove useful in similar situations.

The main ingredient in the proof of Theorem \ref{SA-est} is a decomposition of the representation function
\begin{align} \label{R-def}
R(\mv) := \#\{(y_1,y_2,y_3) \leq N: m_1 = y_1 + y_2 + y_3, m_2 = y_1^2 + y_2^2 + y_3^2 \},
\end{align}
where $\mv \in \N^2, \yv \in \N^3$ and $\#$ denotes the number of elements in the set.
This decomposition will be achieved by using ideas from the circle method and
correspond to a major-minor-arc decomposition of the corresponding exponential sum.
The main work goes into obtaining good control on the arithmetical and analytical 
properties of the different parts, which appear in the decomposition of $R$.

For this task we introduce the triangular function
\begin{align} \label{eq-tri}
\tri(x) = \max\{1-|x|,0\}
\end{align}
and a `continuous indicator function' of $[-P,P]$ with support $[-2P,2P]$ by
\begin{align*}
\psi_P(\alpha) = 2\tri(\alpha/(2P))-\tri(\alpha/P).
\end{align*}
The two-dimensional version that we use is
\begin{align} \label{eq-ind}
\psi(\alv) = \psi_{Q/N}(\alpha_1) \cdot \psi_{Q^2/N^2}(\alpha_2),
\end{align}
where $Q$ and $N$ are the parameters introduced before.

\section{Proof of Theorem \ref{dio-thm}} \label{sec-Beweis}

For a set $\mathcal{A} \subset \{1,2,\ldots, N\}$ write $\einsv_{\mathcal{A}}$ for the indicator function of $\mathcal{A}$ and
$\delta = N^{-1}|\mathcal{A}|$ for the density of $\mathcal{A}$. 
We define the balanced function $f$ by
\begin{align} \label{balanced}
f(n) := \einsv_{\mathcal{A}}(n) - \delta.
\end{align}
Write $Z(N)$ for the number of solutions to \eqref{eq-sys} with $x_i \in \{1,2,\ldots, N\}$ 
and $Z_{\mathcal{A}}(N)$ for the corresponding quantity with $x_i \in \mathcal{A}$.

The proof of Theorem \ref{dio-thm} uses the density increment strategy of Roth \cite{Roth}
and is straightforward, once we have the necessary ingredients.
Let us assume Theorem \ref{SA-est} and the estimate $Z(N) \gg N^{s-3}$ from Proposition \ref{Low}
to give a proof of Theorem  \rm\ref{dio-thm}.

\begin{proof}[Proof of Theorem  \rm\ref{dio-thm}]
The number of solutions $Z_{\mathcal{A}}(N)$ of system \eqref{eq-sys} with variables $x_i \in \mathcal{A}$
can be written as the integral
\begin{align}\label{eq-sys-sol}
Z_{\mathcal{A}}(N) = \int_{\T^2} \prod_{i = 1}^s V_{\mathcal{A}}(\lambda_i\alv) \,d\alv.
\end{align}
Use the decomposition $V_{\mathcal{A}} = \delta V + V_f$ and
expand the product above. We are led to a decomposition into a sum of the main term
\begin{align*}
\delta^s Z(N) = \delta^s \int_{\T^2} \prod_{i = 1}^s V(\lambda_i\alv) \,d\alv
\end{align*}
and $2^s-1$ `error terms' of the form
\begin{align*}
\pm \int_{\T^2} V_f(\lambda_{i_s}\alv) \prod_{j = 1}^{s-1} V_{g_j}(\lambda_{i_j} \alv) \,d\alv,
\end{align*}
where $g_j = \delta$ or $g_j = f$ and $\lambda_{i_1},\ldots,\lambda_{i_s}$ is a permutation of $\{\lambda_1,\ldots,\lambda_s\}$. 
By H\"older's inequality, each error term is bounded up to a constant by
\begin{align} \label{eq-intdec}
H_g = \sup_{\alv} |V_f(\alv)|^{1/2} \int_{\T^2} |V_g(\alv)|^{s-1/2} \,d\alv
\end{align}
for some function $g \in \{\delta, f\}$. The coefficients ($\lambda_i \neq 0$) disappeared by a change of variables 
and the 1-periodicity of the exponential sums.
By Theorem \ref{SA-est} and the trivial bound $|V_g(\alv)| \leq 2 \delta N$ 
the integral in formula \eqref{eq-intdec} is $O(\delta^{s-7} N^{s-1/2-3})$ and hence
\begin{align*}
H_g \ll  \delta^{s-7} N^{s-1/2-3} \sup_{\alv} |V_f(\alv)|^{1/2}.
\end{align*}

Proposition \ref{Low} provides a lower bound $Z(N) \gg N^{s-3}$ for the main term
and H\"older's inequality with Lemma \ref{R2-est} give an upper bound 
$Z_{\mathcal{A}}(N) \ll \delta^{s-7} N^{s-4} \log N$ 
to the number of trivial solutions to the system \eqref{eq-sys}.
(The number of trivial solutions is bounded by $\binom{s}{2}$ times the number of 
solutions to a system with $s-1 \geq 6$ variables.)
Combining the inequalities for $Z(N),Z_{\mathcal{A}}(N),H_g$ and plugging them into 
$\delta^s Z(N) \leq 2^sH_g + Z_{\mathcal{A}}(N)$, we end up with
\begin{align*}
\delta^s N^{s-3} \ll \delta^{s-7} N^{s-1/2-3} \sup_{\alv} |V_f(\alv)|^{1/2} + \delta^{s-7} N^{s-4} \log N.
\end{align*}
A short calculation gives the lower bound
\begin{align*}
\sup_{\alv} |V_f(\alv)| \gg \delta^{14} N
\end{align*}
as long as $N \geq D/\delta^{8}$ for some $D > 1$ (depending on the coefficients $\lambda_i$).
This large Fourier coefficient can be used with Lemma
\ref{exp-prog} to find a progression $P$ of length at least $\delta^{28} N^{1/16}$, such that $\mathcal{A}$ has density
at least $\delta + d\delta^{14}$ on $P$, where $d > 0$ is an absolute constant.
Due to the translation and dilation invariance of the system \eqref{eq-sys}, we end up with the same
problem on a subprogression, but with a slightly higher density.

Since the density is bounded by one, this procedure cannot last more than $d^{-1}\delta^{-14}$ steps 
before reaching a contradiction.
This means that at some stage we have a non-trivial solution or the size of our progression is
getting smaller than $D/\delta^{8}$. The first option is not available by assumption.
Therefore, we have 
\begin{align*}
\delta^{30}N^{(1/16)^{d^{-1}\delta^{-14}}} \leq D/\delta^{8},
\end{align*}
where $D$ depends only on the coefficients of the system.
Rearranging for $\delta$ we can deduce that $\delta \ll 1/\log\log^c(N)$ with $c = 1/15$, for example.
\end{proof}

\begin{remark}
A more careful analysis of the $L^p$-estimates below can yield a better value for $c$.
We haven't optimized because our bound on $|\mathcal{A}|$ is far away from the expected
order of magnitude anyway.
\end{remark}

\section{A general theorem for $L^p$-estimates} \label{sec-LP}

Theorem \ref{Lp-est} below and its proof are strongly connected to 'restriction theory', a subfield
of harmonic analysis studying the behaviour of the Fourier transform under restriction to
a given subset. Recent years have seen several applications of ideas from restriction theory
to number theory, which has led to a better understanding of prime numbers.

The idea for the proof of Theorem \ref{Lp-est} is inspired by the 
papers of Green \cite{Gr} and Green and Tao \cite{GrTao2},
where the connection to restriction estimates can be seen more explicitely. 
The new ingredient in our proof is the use of recursion, which allows us to work under 
slightly weaker assumptions on the involved functions.
Therefore, our result applies to a different range of functions than Proposition 4.2 in \cite{GrTao2}.

For $\Nv = (N_1,\ldots,N_d) \in \N^d$ define $[1,\Nv] := \prod_{i = 1}^d \{n \in \N: n \leq N_i\}$.
Let $\omega : [1,\Nv] \to [0,\infty)$ be a positive real function and $f: [1,\Nv] \to \C$
be a `random' weight function. We study the exponential sums
\begin{align*}
W(\alv) = \sum_{\nv \leq \Nv} \omega(\nv) e(\alv \cdot \nv ) \quad \mbox{ and } \quad
W_f(\alv) = \sum_{\nv \leq \Nv} f(\nv) \omega(\nv) e(\alv \cdot \nv ).
\end{align*}
Consider a general decomposition of $\omega$ and $W(\alv)$ into
\begin{align*}
W(\alv) = \sum_{j \in J} W_j(\alv) \quad \mbox{ and } \quad
\omega(\nv) = \sum_{j \in J} \omega_j(\nv),
\end{align*} 
where $J$ is an index set and $W_j$ is the exponential sum for $\omega_j$.
Define the $L^p$-norms by
\begin{align*}
\|\omega\|_p := \Big(\sum_{\nv \leq \Nv} |\omega(\nv)|^p \Big)^{1/p} \mbox{\ and \ } 
\|W\|_p := \Big(\int_{\T^d} |W(\alv)|^p \,d\alv \Big)^{1/p}_.
\end{align*} 
The typical function $\omega$, of interest for us, obeys $\|\omega\|_1 = O(N)$ but $\|\omega\|^2_2 \neq O(N)$.
One-dimensional examples are the function counting the 
number of representations as sum of two squares and the von Mangoldt function.

\begin{thm} \label{Lp-est}
For $p > 2$ and $\Nv \in \N^d$ we have the estimate
\begin{align*}
\|W_f\|_p \leq \Big(\sum_{\nv \leq \Nv} |f(\nv)|^2 \omega(\nv) \Big)^{1/2} 
\Big(\sum_{j \in J} \|W_j\|_{p}^{(p-2)/p} \|\omega_j\|^{2/p}_{2p/(p-2)} \Big)^{1/2}_.
\end{align*} 
\end{thm}

\begin{remark}
We reduced the estimation of an exponential sum integral with an arbitrary weight function $f$ 
to one expression which involves a weighted $L^2$-norm of $f$, and another one with a decomposition of $\omega$. 
The first factor is easily estimated in our context and the second factor can be written as
\begin{align*}
\left( \sum_{j \in J} \Big( \int_{\T^d} |W_j(\alv)|^p \,d\alv \sum_{\nv \leq \Nv} 
|\omega_j(\nv)|^{2p/(p-2)} \Big)^{(p-2)/p^2} \right)^{1/2}_,
\end{align*} 
where it is easier to see what kind of expressions we need to estimate. 
\end{remark}

\begin{proof}
Write as an abbreviation
\begin{align} \label{Hp-def}
H_p := \int_{\T^d} |W_f(\alv)|^p \,d\alv
\end{align}
for the integral which has to be estimated.
Decompose $|W_f(\alv)|^p = W_f(\alv) \overline{h(\alv)}$ where $h(\alv) = W_f(\alv) |W_f(\alv)|^{p-2}$
and observe the identity
\begin{align} \label{identity}
H_p = \int_{\T^d} |h(\alv)|^{p'} \,d\alv,
\end{align}
where $p' = p/(p-1)$ is the dual exponent of $p$.
We expand the integral \eqref{Hp-def} and apply the Cauchy-Schwarz-inequality to obtain
\begin{align*}
H_p & = \int_{\T^d} |W_f(\alv)|^p \,d\alv = 
\sum_{\nv \leq \Nv} f(\nv) \omega(\nv)  \int_{\T^d} \overline{h(\alv)} e(\alv \cdot \nv) \,d\alv \\
& \leq \Big(\sum_{\nv \leq \Nv} |f(\nv)|^2 \omega(\nv) \Big)^{1/2} 
\Big(\sum_{\nv \leq \Nv} \omega(\nv) \Big| \int_{\T^d} \overline{h(\alv)} e(\alv \cdot \nv) \,d\alv \Big|^2 \Big)^{1/2}_.
\end{align*}
Write $F$ as an abbreviation for the first factor on the right hand side, which is already in the form we want it to be.
By opening the square in the second factor and changing the order of integration and summation, we see that
\begin{align*}
\sum_{\nv \leq \Nv} \omega(\nv) \Big| \int_{\T^d} \overline{h(\alv)} e(\alv \cdot \nv) \,d\alv \Big|^2 =
\int_{\T^d} \int_{\T^d} \overline{h(\alv)} h(\bev) W(\alv - \bev) \,d\bev d\alv.
\end{align*}
This expression is bounded with H\"older's inequality by
\begin{align*}
\Big(\int_{\T^d} |h(\alv)|^{p'} \,d\alv \Big)^{1/p'}
\Big(\int_{\T^d} \Big|\int_{\T^d} h(\bev)  W(\alv - \bev) \,d\bev \Big|^p d\alv \Big)^{1/p}_.
\end{align*}
By \eqref{identity} we can write this as $H_p^{1/p'} \|W * h\|_p$ where the convolution $*$ is defined as usual.
By using the previous formulae and rearranging, we have shown that
\begin{align} \label{eq-Hp}
H^{2-1/p'}_p \leq F^2 \|W * h\|_p.
\end{align}
The decomposition into functions $W_j$ and the triangle inequality imply
\begin{align*}
\|W * h\|_p \leq \sum_{j \in J} \|W_j * h\|_p.
\end{align*}
In the next step, we interpolate to obtain an estimate for the $L^p$-norms from $L^\infty$- and $L^2$-estimates, so that
\begin{align*}
\|W_j * h\|^p_p \leq \|W_j * h\|^{p-2}_{\infty} \|W_j * h\|^2_2.
\end{align*}
This estimate can be seen more easily by writing the left hand side as an integral.
The $L^{\infty}$- and $L^2$-norms are more accessible and
are treated separately in what follows.
For the $L^\infty$-part we have by H\"older's inequality the estimate
\begin{align*}
\|W_j * h\|_{\infty} \leq \|W_j\|_{p} \|h\|_{p'}.
\end{align*}
For the $L^2$-norm we apply Parseval's identity followed by another application of H\"older's inequality
and the Hausdorff-Young inequality, and hence deduce the upper bound
\begin{align*}
\|W_j * h\|_{2} = \|\omega_j \cdot \widehat{h} \|_{2} \leq  \|\omega_j\|_{2p/(p-2)} \|\widehat{h}\|_{p}
\leq  \|\omega_j\|_{2p/(p-2)} \|h\|_{p'}.
\end{align*}
Here $\widehat{h}: \Z^d \to \C$ denotes the Fourier transform of $h$ on $\T^d$.
Putting these estimates in \eqref{eq-Hp} gives
\begin{align*}
H^{2-1/p'}_p
&\leq F^2 \sum_{j \in J} \|W_j * h\|^{(p-2)/p}_{\infty} \|W_j * h\|^{2/p}_2 \\
&\leq F^2 \sum_{j \in J} (\|W_j\|_{p} \|h\|_{p'})^{(p-2)/p} (\|h\|_{p'} \|\omega_j\|_{2p/(p-2)})^{2/p} \\
&\leq F^2 \|h\|_{p'} \sum_{j \in J} \|W_j\|_{p}^{(p-2)/p} \|\omega_j\|^{2/p}_{2p/(p-2)}.
\end{align*}
Now we use \eqref{identity} and $2-2/p'=2/p$ to obtain
\begin{align*}
H_p^{1/p} \leq F \Big(\sum_{j \in J} \|W_j\|_{p}^{(p-2)/p} \|\omega_j\|^{2/p}_{2p/(p-2)} \Big)^{1/2}_.
\end{align*}
\end{proof}

\section{Lemmata}

\emph{More Notation:}  In this section we make use of indicator functions. For a set 
$\mathcal{B}$ we write $\einsv_{\mathcal{B}}$ for the function that is one on $\mathcal{B}$
and zero otherwise. When $\Phi$ is a mathematical statement, then we write $\einsv_{\Phi}$
for the indicator function of the set, where $\Phi$ is true. 

Before we start with the technical core of this work in the next section, 
we collect and prove a few useful lemmata. The reader might want to skip
this section on the first reading and proceed directly to Section \ref{sec-decom}.

The first result summarizes two estimates for the function $R$.

\begin{lem} \label{R2-est}
For $R(\mv)$ as in \eqref{R-def} we have
\begin{align*}
\sum_{\mv \leq 3(N,N^2)} R(\mv)^2 \ll N^3 \log N \quad \mbox{ and } \quad |R(\mv)| \ll_{\epsilon} N^{\epsilon}.
\end{align*}
\end{lem}

\begin{proof}
Rogovskaya \cite{Rog} showed the asymptotic formula 
\begin{align*}
\sum_{\mv \leq 3(N,N^2)} R(\mv)^2 = 18/\pi^2 N^3 \log N + O(N^3).
\end{align*}
But the estimate may be derived from the elementary Lemma 11.2.1 in \cite{Hux} as well.

For the second estimate, Lemma 11.1.1 in \cite{Hux} bounds $R(\mv)$ by nine times the number of divisors of $3m_2-m_1^2$.
By the standard divisor bound, we obtain the second formula as long as $3m_2 \neq m_1^2$.
In the case $m_1^2 = 3m_2$ we can show by elementary calculations that we have $y_1 = y_2 = y_3$ in \eqref{R-def}.
This gives us at most one solution for such a pair $(m_1,m_2)$.
\end{proof}

The next lemma gives us control on the Fourier coefficients of $\psi(\alv) v(\alv)^3$ 
(see \eqref{eq-v} and \eqref{eq-ind}) and
is a crucial tool to estimate the `analytic part' in the decomposition of $R(\mv)$.

\begin{lem} \label{v-est3}
For $\mv \in \N^2$ and $\T^2 = [-1/2,1/2]^2$ we have
\begin{align*}
\int_{\T^2} \psi(\alv) v(\alv)^3 e(-\alv \cdot \mv) \,d\alv \ll 1.
\end{align*}
\end{lem}

\begin{proof}
Recall the definition $v(\alv) = \int_0^N e(\alpha_2t^2 + \alpha_1 t) \,dt$ and 
the decomposition $\psi(\alv) = \psi_{Q/N}(\alpha_1) \cdot \psi_{Q^2/N^2}(\alpha_2)$.
By inserting this into the integral in Lemma \ref{v-est3} and 
changing the order of integration, noting that $\psi(\alv)$ has finite support, we obtain
\begin{align*}
\int_{[0,N]^3} \prod_{i=1}^2 \int_{\R} \psi_{Q^i/N^i}(\alpha_i) e(-\alpha_i (m_i - K_i(\tv))) \,d\alpha_i \,d\tv,
\end{align*}
where we used $K_i(\tv) = t_1^i + t_2^i + t_3^i$ as an abbreviation.
The inner integral is a Fourier transform and our expression can be written as
\begin{align*}
N^3Q^{-3}\int_{[0,N]^3} \prod_{i=1}^2 \widehat{\psi_1}\Big(Q^i N^{-i}(m_i-K_i(\tv))\Big) \,d\tv,
\end{align*}
where $\psi_1 = \psi_P$ with $P = 1$ and $\widehat{\psi_1}$ its Fourier transform.
Change the variables to $x_i = Qt_iN^{-1}$ and rename $y_i = Q^im_iN^{-i}$ to get
\begin{align} \label{eq-psi-K}
\int_{[0,Q]^3} \widehat{\psi_1}(y_1-K_1(\xv)) \widehat{\psi_1}(y_2-K_2(\xv))\,d\xv.
\end{align}
The dependence on $N$ disappeared completely and we are left with the task to bound this integral independent
of $Q$ and $y_i \in \R$.
We use a Riemann sum approach. Consider the level sets 
\begin{align*}
B^i_l := \{\xv \in \R^3: |x_1^i + x_2^i+x_3^i - l| \leq 1/2\}.
\end{align*}
Then the integral \eqref{eq-psi-K} above is bounded by the sum
\begin{align*}
\sum_{l_1 = 0}^\infty \sum_{l_2 = 0}^\infty \mu(B^1_{l_1} \cap B^2_{l_2}) 
\sup_{\substack{|z_1-l_1| \leq 1/2\\ |z_2-l_2| \leq 1/2}} |\widehat{\psi_1}(y_1-z_1)| |\widehat{\psi_1}(y_2-z_2)|,
\end{align*}
where $\mu$ is the Lebesgue measure. The set $B^1_{l_1}$ is a `plane' with thickness 
less than two and $B^2_{l_2}$ is a spherical shell
centered at the origin with `radius' $l_2^{1/2}$ and thickness 
$\sqrt{l_2 + 1} - \sqrt{l_2-1}  \leq 2l_2^{-1/2}$ for $l_2 \geq 1$.
The volume of the intersection of these two objects is maximal for $l_1 = 0$.
It is a volume around a circle with radius $O(l_2^{1/2})$. The width in the radial direction
is $O(l_2^{-1/2})$ and bounded by absolute constants otherwise. 
Therefore, the volume $B^1_{l_1} \cap B^2_{l_2}$ is bounded by a constant independent of $l_2$ or $l_1$.
We are left with the product of two sums
\begin{align*}
\Big(\sum_{l_1 = 0}^\infty \sup_{|z_1-l_1| \leq 1/2} |\widehat{\psi_1}(y_1-z_1)| \Big)
\Big(\sum_{l_2 = 0}^\infty \sup_{|z_2-l_2| \leq 1/2}  |\widehat{\psi_1}(y_2-z_2)| \Big).
\end{align*}
The Fourier transform of $\psi_1= 2\tri(\alpha/2)-\tri(\alpha)$ can be understood by using
the well known identity $\widehat{\tri}(x) = \frac{\sin^2(\pi x)}{(\pi x)^2}$.
The sums are easily seen to be convergent with an upper bound independent of the $y_i$.
\end{proof}

The idea for the next lemma is borrowed from \cite{GrTao2} and is the arithmetic counterpart to
Lemma \ref{v-est3}.

\begin{lem} \label{2g-est}
Let $g: \N^2 \to \C$ be a function with $|g(a,q)| \leq Cq^{-1}$ for some $C \geq 1$.
Consider the function $\beta: \N \to \C$, given by

\begin{align*}
\beta(m) = \sum\limits_{q \leq X} \sum\limits_{a \leq q} g(a,q) e_q(-am).
\end{align*}
Then for $k \in \N$ and $X^{4k} \leq M$ one has
\begin{align*}
\sum\limits_{m \leq M} |\beta(m)|^{2k} \ll_{\epsilon,k} M X^{\epsilon}.
\end{align*}
\end{lem}

\begin{proof}
First, we insert the definition of $\beta$ and expand to obtain
\begin{align} \label{eq-Masympt}
\sum_{m \leq M} |\beta(m)|^{2k} =
\sum_{q_1,\ldots,q_{2k} \leq X} \sum_{a_1,\ldots,a_{2k} \leq q} 
\Big(\prod_{j \leq 2k} g_j(a_j,q_j) \Big)
\sum_{m \leq M} e\Big(m \sum_{i \leq 2k} \epsilon_i a_i/q_i \Big),
\end{align}
where $g_j = g$ and $\epsilon_i = -1$ for $i,j \leq k$, and $g_j = \overline{g}$ and $\epsilon_i = 1$ 
for $k < i,j \leq 2k$. Write $\Phi$ for the condition
\begin{align*}
\Big(\sum_{i \leq 2k} \epsilon_i a_i/q_i \Big) \in \Z.
\end{align*} 
Evaluating the innermost sum in \eqref{eq-Masympt}, we get the 
expression $M \einsv_{\Phi} + O(q_1\cdots q_{2k})$.
Together with the estimate $|g(a,q)| \leq Cq^{-1}$ this gives
\begin{align*}
\sum_{m \leq M} |\beta(m)|^{2k}
\ll  \sum_{q_1,\ldots,q_{2k} \leq X} \sum_{a_1,\ldots,a_{2k} \leq q} \frac{M}{q_1 \cdots q_{2k}}
\einsv_{\Phi} + O(X^{4k}).
\end{align*}
The function $\einsv_{\Phi}$ can now be written as exponential sum again
and we gain the original expression
\begin{align*}
\sum_{m \leq M} |\beta(m)|^{2k} \ll \sum_{m \leq M} \Big|\sum_{q \leq X} q^{-1} 
\sum_{a \leq q}  e_q(-am)\Big|^{2k} + O(X^{4k}),
\end{align*}
but now with $g(a,q)$ replaced by $q^{-1}$.
The exponential sum inside is $0$, except when $q$ divides $m$. In that case it gives $q$
and we are left with the task to estimate
\begin{align*}
\sum_{m \leq M} \tau_X(m)^{2k},
\end{align*}
where $\tau_X(m) = \sum_{q \leq X} \einsv_{q|m}$ is a restricted divisor function.
By expanding again, we have
\begin{align*}
\sum_{m \leq M} \tau_X(m)^{2k} &= \sum_{m \leq M} \Big(\sum_{q \leq X} \einsv_{q|m} \Big)^{2k}
= \sum_{q_1,\ldots,q_{2k} \leq X} \sum_{m \leq M} \einsv_{q_1|m,\ldots,q_{2k}|m}\\
&\leq \sum_{q_1,\ldots,q_{2k} \leq X} \frac{M}{\lcm(q_1,\ldots,q_{2k})}
\leq  M \sum_{l \leq X^{2k}} \frac{\tau(l)^{2k}}{l}.
\end{align*}
The last inequality follows from the fact that the equation $l = \lcm(q_1,\ldots,q_{2k})$ 
has at most $\tau(l)^{2k}$ solutions where $\tau$ is the usual divisor function.
Using the standard estimate $\tau(l) \ll_{\eta} l^{\eta}$ in the range $l \leq X^{2k}$ and
choosing $\eta = \epsilon/5k^2$, we get
\begin{align*}
\sum_{l \leq X^{2k}} \tau(l)^{2k}l^{-1} \ll_{k,\epsilon} ((X^{2k})^{\eta})^{2k} \sum_{l \leq X^{2k}} l^{-1}
\ll_{k,\epsilon} X^{\epsilon}.
\end{align*}
\end{proof}

The slightly technical Lemma \ref{transform} is needed to transform 
a two-dimensional exponential sum into a one-dimensional version.

\begin{lem} \label{transform}
For fixed $q \in \N$ and $\mv \in \N^2$ we can write
\begin{align} \label{Lem3.10eq}
\sum_{(\av;q) = 1} V(q,\av)^3 e_q(-\av \cdot \mv) 
= q \sum_{(a;q) = 1} G_{m_1}(q,a)\, e_q(-am_2),
\end{align}
where 
\begin{align*}
G_{m_1}(q,a) = \sum_{r_1,r_2 \leq q}  e_q(a F_2(r_1,r_2) + a F_{1,m_1}(r_1,r_2))
\end{align*}
with $F_2(r_1,r_2) = 2r_1^2+ 2r_1r_2 +2r_2^2 $ and $F_{1,m_1}(r_1,r_2) = - 2m_1(r_1+r_2) + m_1^2$.
\end{lem}

\begin{remark}
The exact form of the polynomial above is not very important. The only fact we need is that $F_2$
is a non-degenerate quadratic form.
\end{remark}

\begin{proof}
Use the well known identity $\einsv_{t = 1} = \sum_{d|t} \mu(d)$ to rewrite the condition $(\av;q) = 1$
on the left hand side $\mathcal{L}$ of equation \eqref{Lem3.10eq}. We get
\begin{align*}
\mathcal{L} = \sum_{(\av;q) = 1} V(q,\av)^3 e_q(-\av \cdot \mv) 
= \sum_{\av \leq q} \sum_{d|(\av;q)} \mu(d) V(q,\av)^3 e_q(-\av \cdot \mv).
\end{align*}
Next we observe that $d|(\av;q)$ if and only if $d|q$ and $d|\av$. Rearranging gives
\begin{align*}
\mathcal{L} = \sum_{d|q} \mu(d) \sum_{d|\av}  V(q,\av)^3 e_q(-\av \cdot \mv)
\end{align*}
Note that the first sum is over all $d$ dividing $q$ and the second 
sum runs over all $\av \leq q$ with components divisible by $d$.
Now we look closer at the inner sum and insert the definition $V(q,\av) = \sum_{r =1}^q e_q(a_2r^2 + a_1r)$.
Using the abbreviation $K_i(\rv) = r_1^i + r_2^i + r_3^i$, we obtain the identity
\begin{align*}
\mathcal{L} = \sum_{d|q} \mu(d) \sum_{\rv \leq q} \sum_{d|a_2} e_q(a_2 (K_2(\rv)-m_2)) \sum_{d|a_1} e_q(a_1 (K_1(\rv)-m_1)).
\end{align*}

An evaluation of the two exponential sums gives
\begin{align*}
\mathcal{L}  = \sum_{d|q} \mu(d) q^2d^{-2} \sum_{\rv \leq q} \einsv_{\rv,\mv,q/d},
\end{align*}
where $\einsv_{\rv,\mv,q/d}$ is indicator function with conditions
\begin{align*}
K_2(\rv)-m_2 \equiv 0 \pmod{q/d} \quad\mbox{ and }\quad K_1(\rv)-m_1 \equiv 0 \pmod{q/d}.
\end{align*}

Recalling the definitions of $K_i(\rv)$, we can insert the second congruence into the first.
We arrive at the relation
\begin{align} \label{eq-rel}
r_1^2 + r_2^2 + (m_1-r_1-r_2)^2 -m_2 \equiv 0 \pmod{q/d}.
\end{align}
Since $r_3$ disappeared from the first condition, we can perform the sum over $r_3$, 
which reduces to the evaluation
\begin{align*}
\sum_{r_3 \leq q} \einsv_{[K_1(\rv)-m_1 \equiv 0 \pmod{q/d}]} = d.
\end{align*}
Therefore, by rearranging \eqref{eq-rel}, we get
\begin{align*}
\mathcal{L} = q \sum_{d|q} \mu(d) q d^{-1} \sum_{r_1,r_2 \leq q} 
\einsv_{[H(r_1,r_2,m_1,m_2) \equiv 0 \pmod{q/d}]},
\end{align*}
where $H(r_1,r_2,m_1,m_2) = F_2(r_1,r_2) + F_{1,m_1}(r_1,r_2) -m_2$. 
Then we can rewrite the indicator function as an exponential sum and get
\begin{align*}
\mathcal{L} = q \sum_{d|q} \mu(d) \sum_{r_1,r_2 \leq q}  \sum_{d|a} e_q(aH(r_1,r_2,m_1,m_2)).
\end{align*}
Changing the order of the sums and using the identity with the M\"obius function again
gives the right hand side of the equation in Lemma \ref{transform}.
\end{proof}

\section{Decomposition of $R$ and proof of Theorem \ref{SA-est}} \label{sec-decom}

In this section we perform a decomposition of $R(\mv)$ and give estimates for its analytic and
arithmetic parts in Propositions \ref{L28-est} and \ref{arith-est}.
In combining these propositions with Theorem \ref{Lp-est}, we are able to deduce Theorem \ref{SA-est}.

To motivate the technical material below, 
let us first look at the structure of the proof of Theorem \ref{SA-est}.

\begin{proof}[Proof of Theorem \rm\ref{SA-est}.]
We apply Theorem \ref{Lp-est} with $\omega(\mv) = R(\mv)$. 
Let $\Nv = 3(N,N^2)$ and define $W_h(\alv)$ by
\begin{align*}
V_g(\alv)^3 = \sum_{\mv \leq 3(N,N^2)} h(\mv) R(\mv) e(\alv \cdot \mv)
\end{align*}
where $|h(\mv)| \leq 1$ is chosen in such a way that 
\begin{align*}
h(\mv) R(\mv) = \sum_{\substack{x_1,x_2,x_3 \leq N\\m_1 = x_1 + x_2 + x_3 \\ 
m_2 = x_1^2 + x_2^2 + x_3^2}} g(x_1)g(x_2)g(x_3).
\end{align*}
Then we have
\begin{align*}
\sum_{\mv \leq 3(N,N^2)} |h(\mv)|^2 R(\mv) \leq \sum_{\mv \leq 3(N,N^2)} R(\mv) = N^3
\end{align*}
for the first term in Theorem \ref{Lp-est}.

Let $J = \{1,2,4,\ldots,2^{D-1}, 2^D\}$ with $D \approx \log_2 Q$.
To estimate the second term, we need a suitable decomposition of $R(\mv)$.
Later in this section we construct the decomposition
\begin{align*}
R(\mv) =  \sum_{Y \in J} R_{Y}(\mv) + R'(\mv)
\end{align*}
and for the corresponding exponential sum $W(\alv) = V(\alv)^3$ similarly
\begin{align*}
W(\alv) =  \sum_{Y \in J} W_Y(\alv) + W'(\alv).
\end{align*}
Write $J'$ for the index set $J \cup \{0\}$ and write $W_0 = W'$ and $R_0 = R'$
to simplify notation.
Since $p > 6$, we have $b = p/3 > 2$ and Theorem \ref{Lp-est} gives us the estimate
\begin{align*}
& \int_{\T^2} |V_g(\alv)|^p \,d\alv = \int_{\T^2} |W_h(\alv)|^b \,d\alv\\
\leq\ & 
N^{3b/2}\left( \sum_{Y \in J'} \Big( \int_{\T^2} |W_Y(\alv)|^{b} \,d\alv \sum_{\nv \leq \Nv} 
|R_Y(\nv)|^{2b/(b-2)} \Big)^{(b-2)/b^2} \right)^{b/2}_.
\end{align*}
The necessary moment estimates to handle this expression are given by Proposition \ref{L28-est}
and Proposition \ref{arith-est} below.
(The exponent $2b/(b-2)$ of $R_Y$ might not be of the form $2k$ with $k \in \N$ as in Proposition \ref{arith-est},
but it is easily seen by H\"older's inequality to be applicable nevertheless if we choose $Q^{4b/(b-2)+1} \leq N$.) 
We use Proposition \ref{L28-est} and the estimate
\begin{align*}
\int_{\T^2} |W_Y(\alv)|^{b} \,d\alv \ll \sup_{\alv \in \T^2} |W_Y(\alv)|^{b-2} \int_{\T^2} |W_Y(\alv)|^{2} \,d\alv
\end{align*} 
to bound the integral over $W_Y$.
Summing over $Y$ and singling out the term with $Y = 0$, we get the upper bound
\begin{align*}
\ll_{b,\epsilon} N^{3b/2} \Big( \sum_{i = 0}^{D} \left( 2^{(-3(b-2)/2 + \epsilon)i} N^{3b}  \right)^{(b-2)/b^2}
+ \left( Q^{2-b} N^{3b+\epsilon}\right)^{(b-2)/b^2} \Big)^{b/2}_,
\end{align*} 
where we already absorbed the logarithms into the $N^{\epsilon}$ term.
For $\epsilon = \epsilon(b)$ small enough, the sum over $i$ is convergent and
if $Q$ is a small power of $N$, depending on $b = p/3$, another short calculation leaves us with the result
\begin{align*}
\int_{\T^2} |V_g(\alv)|^p \,d\alv \ll_{b} N^{3b-3} \ll_p N^{p-3}.
\end{align*}
\end{proof}

To fill in the gaps in the proof of Theorem \ref{SA-est}, we need to construct a suitable decomposition of $R(\mv)$.
We start with an auxiliary function $U_Y(\alv)$, that is motivated by Lemma \ref{Vau-thm}.
Use the cutoff-function $\psi$ to restrict the major arc approximation from Lemma \ref{Vau-thm}
to a $(Q/N,Q^2/N^2)$-neighbourhood of $\av/q$. Define for $Y \leq 2Q$
\begin{align} \label{eq-Uj}
U_Y(\alv) := \sum_{Y \leq q < 2Y} q^{-3} \sum_{(\av;q) = 1} V(q,\av)^3
 v(\alv-\av/q)^3 \psi(\alv-\av/q)
\end{align}
and consider $U_Y(\alv)$ as a function on 
\begin{align*}
\T^2 = [2QN^{-1}, 1 + 2QN^{-1}] \times [2Q^2 N^{-2}, 1 + 2Q^2 N^{-2}]
\end{align*}
(see notational conventions).
The functions $U_Y(\alv)$ arise from the main term in the major arc approximation of $V(\alv)$
for denominators $Y \leq q < 2Y$.
It is a sum of disjointly supported pieces due to the restriction $64 \, Q^{52} \leq N$. 
Take the corresponding arithmetical functions
\begin{align} \label{eq-Rsum}
R_Y(\mv) =  \int_{\T^2} U_Y(\alv) e(-\alv \mv) \,d\alv,
\end{align}
but restrict to $1 \leq \mv \leq 3(N,N^2)$. Set $R_Y(\mv) = 0$ for other values of $\mv$.


Let $D$ be an integer between $\log_2 Q$ and $\log_2 Q + 1$, where $\log_2$ is
the logarithm to base two and set $J = \{1,2,4,\ldots,2^{D-1}, 2^D \}$. We write
\begin{align*}
R(\mv) =  \sum_{Y \in J} R_{Y}(\mv) + R'(\mv).
\end{align*}
This is the decomposition used in the proof of Theorem \ref{SA-est} above.

\begin{remark}
The part $R'(\mv)$ may be thought of as corresponding to the minor arcs,
but it also contains the approximation error on the major arcs.
\end{remark}

Write $W(\alv)$ for
\begin{align*}
V^3(\alv) = \sum_{\mv \leq 3(N,N^2)} R(\mv) e(\alv \cdot \mv),
\end{align*} 
and define the corresponding exponential sums for $R_Y(\mv)$ as
\begin{align*}
W_Y(\alv) = \sum_{\mv \leq 3(N,N^2)} R_Y(\mv) e(\alv \cdot \mv).
\end{align*}
By inserting the definition of $R_Y(\mv)$ we see that
\begin{align} \label{eq-Tint}
W_Y(\alv) = \int_{\T^2} U_Y(\bev) L(\alv-\bev) \,d\bev,
\end{align}
where $L$ is the linear exponential sum defined in Section \ref{notation}.
As in the case of $R(\mv)$, we obtain a decomposition
\begin{align}   \label{eq-Wdec}
W(\alv) =  \sum_{Y \in J} W_Y(\alv) + W'(\alv),
\end{align}
where $W'$ corresponds to $R'$.
Now we are ready to give the first proposition that was needed for the proof of Theorem \ref{SA-est}.  

\begin{prop} \label{L28-est}
For $Y \leq 2Q$ we have the estimates
\begin{align*}
\int_{\T^2} |W_Y(\alv)|^2 \,d\alv \ll  N^3, \qquad \int_{\T^2}|W'(\alv)|^2 \,d\alv\ll N^{3}\log^2 N,
\end{align*}
\begin{align*}
\sup_{\alv} |W_Y(\alv)| \ll Y^{-3/2} N^3, \qquad \sup_{\alv} |W'(\alv)| \ll N^{3}Q^{-1} \log^2 N.
\end{align*}
\end{prop}

\begin{proof}
By inserting \eqref{eq-Tint} into the first interal, we obtain
\begin{align*}
\int_{\T^2} |W_Y(\alv)|^2 \,d\alv 
= \int_{\T^6}U_Y(\bev) \overline{U_Y(\gamv)}  L(\alv-\bev) \overline{L(\alv-\gamv)}\,d\bev d\gamv d\alv.
\end{align*}
The integration with respect to $\alv$ can be performed explicitly by writing out the definition of 
the linear exponential sum $L$ and using orthogonality.
Changing the order of summation and integration, we arrive at
\begin{align*}
\int_{\T^2} |W_Y(\alv)|^2 \,d\alv  = \sum_{\mv \leq 3(N,N^2)} \Big|\int_{\T^2} U_Y(\bev) e(-\bev \cdot \mv)\,d\bev \Big|^2.
\end{align*}
By the Bessel-inequality on the Hilbert space $L^2(\T^2)$, we can bound this
by $\|U_Y\|^2_2$. We insert the definition \eqref{eq-Uj} of $U_Y(\alv)$ and recall that the functions 
in this sum have disjoint supports. We obtain
\begin{align*}
\|U_Y\|^2_2 = \sum_{Y \leq q < 2Y} q^{-6}  \sum_{(\av;q) = 1} |V(q,\av)|^6 
\int_{\T^2} |v(\bev - \av/q)|^6 \psi^2(\bev - \av/q)\,d\bev.
\end{align*}
This is $O(N^3)$ by Lemma \ref{v-est2} and Lemma \ref{exp-est}.

We proceed with the corresponding bound for $W'$ and observe that by \eqref{eq-TUsum} we have
\begin{align*}
\int_{\T^2} |W'(\alv)|^2 \,d\alv \ll \int_{\T^2} |W(\alv)|^2 \,d\alv 
+ \int_{\T^2} \Big|\sum_{Y \in J}W_Y(\alv)\Big|^2 \,d\alv.
\end{align*}
By Lemma \ref{R2-est}, the bound for the first integral on the right hand side is
\begin{align*}
\int_{\T^2} |W(\alv)|^2 \,d\alv = \sum_{\mv \leq 3(N,N^2)} R(\mv)^2 \ll N^3 \log N.
\end{align*} 
To estimate the second integral, we use the estimates for $W_Y(\alv)$ and the Cauchy-Schwarz-inequality
for the sum over $Y \in J$. This leads to a bound of $O(N^{3}\log^2 N)$.

Now we turn to the second part of the proposition concerning the $L^{\infty}$-estimates. 
Using \eqref{eq-Tint} we obtain for $W_Y(\alv)$ the bound
\begin{align*}
|W_Y(\alv)| \leq \int_{\T^2} |U_Y(\bev)| |L(\alv-\bev)| \,d\bev.
\end{align*}
Using definition \eqref{eq-Uj} of $U_Y(\alv)$ and Lemma \ref{exp-est} again, this gives the bound
\begin{align*}
|W_Y(\alv)| & \leq  \sum_{Y \leq q < 2Y} q^{-3} \sum_{(\av;q) = 1} |V(q,\av)|^3
\int_{\T^2} |v(\bev - \av/q)|^3 \psi(\bev - \av/q) |L(\alv-\bev)|\,d\bev \\
& \ll  Y^{-3/2} \sum_{Y \leq q < 2Y} \sum_{(\av;q) = 1}
\int_{\T^2} |v(\bev - \av/q)|^3 \psi(\bev - \av/q) |L(\alv-\bev)|\,d\bev.
\end{align*}
It remains to estimate the inner integral in dependence on $\av$ and $q$ 
in such a way that the sum is of size $O(N^3)$.
For fixed $\alv$, there is at most one triple $(q,\av)$ such that 
$\|\alv-\av/q\| \leq (2q)^{-2}$. Otherwise, we would get
\begin{align*}
(qq')^{-1} \leq |a_1/q-a'_1/q'| \leq (2q)^{-2} + (2q')^{-2},
\end{align*} 
which is impossible if $q$ and $q'$ differ by a factor of at most $2$.
For this triple we apply the Cauchy-Schwarz-inequality and obtain
\begin{align*}
& \int_{\T^2} |v(\bev - \av/q)|^3 \psi(\bev - \av/q) |L(\alv-\bev)|\,d\bev \\
\leq & \Big( \int_{\T^2} |v(\bev - \av/q)|^6 \psi^2(\bev - \av/q) \,d\bev \Big)^{1/2}
\Big( \int_{\T^2} |L(\alv-\bev)|^2\,d\bev \Big)^{1/2}.
\end{align*}
This is $O(N^3)$ by Lemma \ref{v-est2} and orthogonality.
For the other triples $(q,\av)$ we use the trivial estimates $|L(\alpha)| \leq \|\alpha\|^{-1}$
and $|v(\bev - \av/q)| \leq N$. This leaves us with the bound for the remaining part of the form
\begin{align*}
N^3 \sum_{Y \leq q < 2Y} \sum_{\substack{(\av;q) = 1}}^*
\int_{\T^2} \psi(\bev - \av/q) \|\alpha_1-\beta_1\|^{-1} \|\alpha_2-\beta_2\|^{-1}  \,d\bev.
\end{align*}
The $*$ at the sum indicates that we are leaving out the (possibly existing) triple $(q,\av)$ with
$\|\alv-\av/q\| \leq (2q)^{-2}$.
This implies that $\alv$ is at least $(3q)^{-2}$ apart from the support of any remaining
$\psi(\bev - \av/q)$ and $\|\alpha_i-\beta_i\|^{-1}$ can be bounded by $O(Q^2)$.
The support of $\psi$ is of size $O(Q^3N^{-3})$, which allows us to use the crude estimate 
$\int_{\Omega} f(x) \,d\mu \leq \mu(\Omega) \sup_{x \in \Omega} |f(x)|$ for each summand.
It leads to a bound of size $O(Q^{10})$ for the whole expression, which is far better than needed.

Now we come to the last part of the proof, the $L^{\infty}$-estimate for $W'$.
Write 
\begin{align} \label{eq-TUsum}
U^*(\alv) = \sum_{Y \in J} U_Y(\alv)
\end{align} 
to simplify notation. We can rewrite the identity $W(\alv) = V^3(\alv)$ in a slightly complicated way as
\begin{align*}
W(\alv) =  \int_{\T^2} V^3(\bev) L(\alv-\bev) \,d\bev,
\end{align*}
as can be seen by calculating the right hand side explicitly.

By Lemma \ref{lin-est} and equations \eqref{eq-Wdec} and \eqref{eq-Tint}, we can write
\begin{align*}
|W'(\alv)| \leq  \int_{\T^2} |V^3(\bev) - U^*(\bev)| |L(\alv-\bev)| \,d\bev
\leq \|V^3 - U^*\|_{\infty} \log^2 N
\end{align*}
where the $L^\infty$-norm is taken over 
$\T^2 = [2QN^{-1}, 1 + 2QN^{-1}] \times [2Q^2 N^{-2}, 1 + 2Q^2 N^{-2}]$.

If we are outside the support of $U^*$, we have $|V(\alv)| \ll NQ^{-1/3}$ by Lemma $\ref{ma-est}$.
This leads to the bound $O(N^3Q^{-1}\log^2 N)$.
If $\alv$ is in the support of $U^*$ (major arc case), we have
an approximation $\alpha_i = a_i/q + \beta_i$ with $q \leq 4Q$ and $|\beta_i| \leq 2Q^i N^{-i}$,
where $\av,q$ and $\bev$ are defined as in Lemma \ref{Vau-thm}. 
By \eqref{eq-Uj}, \eqref{eq-TUsum} and the choice of the parameter $Q$ we see that $U^*$ is
a sum of disjointly supported functions. 
Unsing the approximation for $V(\alv)$ from Lemma \ref{Vau-thm}, we obtain the expression 
\begin{align*}
V^3(\alv) - U^*(\alv) = (q^{-1} V(q,\av) v(\bev) + \Delta)^3 - q^{-3} V(q,\av)^3 v(\bev)^3 \psi(\bev).
\end{align*}
The trivial estimate is $O(N^3)$ and the task is to save a factor of $Q$.
We have $\Delta \ll Q^3$ and can safely ignore terms involving $\Delta$ as long as $Q^3 \leq NQ^{-1}$,
which is true by our choice of $Q$. We are left with the task to estimate 
\begin{align*}
q^{-3} V(q,\av)^3 v(\bev)^3 (1-\psi(\bev)^3) \leq v(\bev)^3 (1-\psi(\bev)^3).
\end{align*}
This is zero if $\psi(\bev) = 1$. Otherwise, we have $|\beta_i| \geq Q^i/N^i$ for some $i \in \{1,2\}$ by \eqref{eq-ind}
and Lemma \ref{v-est} provides $|v(\bev)| \ll NQ^{-1/2}$.
This gives an estimate of the shape $O(N^3Q^{-3/2})$ which is acceptable as long as $Q$ is a small power of $N$.
\end{proof}

The next part is concerned with the arithmetic estimates.
\begin{prop} \label{arith-est}
For each $k \in \N$ with $Q^{4k} \leq N$ we have
\begin{align*} 
\sum_{\mv \leq 3(N,N^2)} |R_Y(\mv)|^{2k} \ll_{\epsilon,k} Y^{\epsilon} N^3 \mbox{\quad and \quad} 
\sum_{\mv \leq 3(N,N^2)} |R'(\mv)|^{2k} \ll_{\epsilon,k} N^{3+\epsilon}.
\end{align*}
\end{prop}

\begin{remark}
Notice that we can choose $k$ as big as we like by reducing $Q$.
\end{remark}

\begin{proof}
We plug in formula \eqref{eq-Uj} into \eqref{eq-Rsum}. This gives an expression for $R_Y$
which we transform by changing the order of summation and integration and the change of 
variables $\alv = \av/q + \bev$ into
\begin{align*} 
R_Y(\mv) =  \sum_{Y \leq q < 2Y} \frac{1}{q^3} \sum_{(\av;q) = 1} V(q,\av)^3 e_q(-\av \cdot \mv)
\int_{\T^2} \psi(\bev) v(\bev)^3 e(-\bev \mv) \,d\bev.
\end{align*}
This form is more suitable for estimates, because the arithmetic part 
and the analytic part are separated.
By Lemma \ref{v-est3}, the integral, appearing in the formula for $R_Y$, may be estimated by a constant.
Using Lemma \ref{transform}, we can write the arithmetic part in the form 
\begin{align*}
\frac{1}{q^3}  \sum_{(\av;q) = 1} V(q,\av)^3 e_q(-\av \cdot \mv) 
=  \frac{1}{q^2} \sum_{\substack{a=1 \\ (a;q) = 1}}^q G_{m_1}(q,a)\, e_q(-am_2).
\end{align*}
Lemma \ref{exp-est} provides the estimate $q^{-2}G_{m_1}(q,a) \ll q^{-1}$ if $(a;q) = 1$ and 
we set $G_{m_1}(q,a) = 0$ if $(a;q) > 1$. By Lemma \ref{2g-est}, we have
\begin{align*} 
\sum_{m_2 \leq 3N^2} |R_Y(\mv)|^{2k} \ll_{\epsilon,k} Y^{\epsilon} N^2
\end{align*}
as long as $Q^{4k} \leq N$, which implies the first result by summing over $m_1$.

For the second estimate, we observe that $R'(\mv) = R(\mv) - \sum_{Y \in J} R_Y(\mv)$.
The result follows from an application of H\"older's inequality, part one of this proof and Lemma \ref{R2-est}.
\end{proof}

\section{Finding integer solutions of the system} \label{sec-Circle}

This is the part of our paper which has not changed significantly compared
to the work of Smith \cite{Smith}. We use the circle method to prove the following proposition, 
which together with Theorem 2.1 concludes the proof of Theorem 1.1.
Write $Z(N)$ for the number of solutions to \eqref{eq-sys} with $x_i \in \{1,\ldots,N\}$.

\begin{prop} \label{Low}
Assume that the conditions of Section \ref{introduction} hold. Then $Z(N) \gg N^{s-3}$, where
the implied constant is only dependent on the coefficients $\lambda_i$ of the system.
\end{prop}

By orthogonality, we can rewrite $Z(N)$ as the integral
\begin{align} \label{eq-intrep}
Z(N) = \int_{\T^2} \prod_{i = 1}^s V(\lambda_i\alv) \,d\alv.
\end{align}
Define the major arcs $\Ma \subset \T^2$ to be the (disjoint) union of 
\begin{align*}
\Ma(q,\av) = \{\alv \in \T^2: \|\alpha_i-a_i/q_i\| \leq Q^iN^{-i}\},
\end{align*}
where $q \leq Q$ and $\av \leq q$ with $(\av;q) = 1$.
The minor arcs are the complement $\ma = \T^2\backslash \Ma$.

We decompose the right hand side of \eqref{eq-intrep} into an integral over $\Ma$ and one over $\ma$.
By using Lemma \ref{Vau-thm} and the crude estimates  
$|V(q,\lambda_i\av)| \leq q$, $|v(\alpha)| \leq N$, $|\Delta| \leq Q^{3}$ and the bound 
$\mu(\Ma(q,\av)) \leq Q^3/N^{3}$ for the Lebesgue measure of $\Ma$, we obtain by a standard calculation
\begin{align*}
\int_{\Ma} \prod_{i = 1}^s V(\lambda_i\alv) \,d\alv = J(Q/N) \Sing(Q) + O(Q^9 N^{s-4}),
\end{align*}
with the truncated singular integral 
\begin{align*}
J(Q/N) = \int_{|\alpha_1| \leq Q/N}\int_{|\alpha_2| \leq Q^2/N^2} \prod_{i = 1}^s v(\lambda_i \alv) \, d\alv,
\end{align*}
and the truncated singular series
\begin{align} \label{eq-sing}
\Sing(Q) = \sum_{q \leq Q} q^{-s} \sum_{(\av;q) = 1} \prod_{i = 1}^s V(q,\lambda_i\av).
\end{align}
It follows from Lemma \ref{exp-est} and the condition $s \geq 7$ that the limit 
$\Sing := \lim_{Q \to \infty} \Sing(Q)$ exists and we have $\Sing(Q) = \Sing + O(Q^{-1/2})$.
Similarly, we complete $J(Q/N)$ to an integral
\begin{align} \label{eq-J}
J = \int_{\R^2} \prod_{i = 1}^s v(\lambda_i \alv) \, d\alv
\end{align}
over $\R^2$. Convergence and the formula $J(Q/N) = J + O(N^{s-3}Q^{-1/2})$ follow by Lemma \ref{v-est}.
This gives a major arcs contribution of
\begin{align*}
\int_{\Ma} \prod_{i = 1}^s V(\lambda_i \alv) \, d\alv = \Sing J + O(Q^9N^{s-4} + N^{s-3}Q^{-1/2}).
\end{align*}
For $\alv \in \ma$ we observe the bound $O(NQ^{-1/3})$ for the $L^\infty$-norm of $V(\alv)$ by Lemma $\ref{ma-est}$.
The actual exponential sums $V(\lambda_i \alv)$ can be bounded similarly by $O(NQ^{-1/3})$, 
where the implied constant is now dependent on the fixed coefficient $\lambda_i \in \Z$.
Using H\"older's inequality and the bound from Lemma \ref{R2-est} for the sixth moment of the exponential sum, 
we obtain
\begin{align*}
\int_{\ma} \prod_{i = 1}^s V(\lambda_i \alv) \, d\alv 
\ll (NQ^{-1/3})^{s-6} \int_{\T^2} |V(\alv)|^6 \, d\alv \ll_{\epsilon} N^{s-3+\epsilon}Q^{-1/3}.
\end{align*}
Setting $Q= N^{4\eta}$ for a small $\eta > 0$, the number of solutions is given by
\begin{align*}
Z(N) = \Sing J + O(N^{s-3-\eta}).
\end{align*}
If we can show that $\Sing > 0$ and $J \gg N^{s-3}$, then Proposition \ref{Low} is proven.
These are the remaining tasks for this section.

\begin{lem} \label{lem-sing}
We have $\Sing > 0$.
\end{lem}

\begin{proof}
We denote the summand in \eqref{eq-sing} by
\begin{align*}
V(q) = q^{-s} \sum_{(\av;q) = 1} \prod_{i = 1}^s V(q,\lambda_i\av).
\end{align*}
It is not too difficult to check that $V(qt,\lambda_it\av+\lambda_iq\bv) = V(q,\lambda_i\av)V(t,\lambda_i\bv)$
as long as $(t;q) = 1$.
This implies that $V(q)$ is multiplicative by another short calculation.
For details, the reader is referred to \cite[p. 20]{Vau}.
Furthermore $V(q) \ll q^{-3/2}$ by Lemma \ref{exp-est} and $s \geq 7$. If we define
\begin{align*}
T(p) := \sum_{k=0}^{\infty} V(p^k)
\end{align*}
as the part corresponding to the prime $p$, then $T(p) = 1 + O(p^{-3/2})$ 
and we can write $\Sing$ as an absolutely convergent product $\Sing = \prod_p T(p)$.
Therefore, we have $\prod_{p \geq p_0} T(p) \geq 1/2$ for some $p_0$.
To show that $T(p) > 0$ for the finitely many remaining factors, we relate
$T(p)$ to the number of solutions modulo $p^k$ and the result follows from the following two lemmata.
\end{proof}

\begin{lem} \label{p-adic}
We have
\begin{align*}
T(p) = \lim_{k \to \infty} p^{(2-s)k} S(p^k)
\end{align*}
where $S(q)$ is the number of solutions to the system \eqref{eq-sys} modulo $q$.
\end{lem}

\begin{proof} By orthogonality
and a direct calculation, the number of solutions $S(q)$ can be written as
\begin{align*}
S(q) = \sum_{r_1,\ldots,r_s \leq q} q^{-2} \sum_{a_1,a_2 \leq q} 
e_q(a_1(\lambda_1r_1^2 + \ldots + \lambda_sr_s^2)) e_q(a_2(\lambda_1r_1 + \ldots + \lambda_sr_s)).
\end{align*}
Here we are in the special situation of $q = p^k$. We can introduce another parameter $l$ and sort
according to the condition $(a_1;a_2) = p^l$. After rearrangement, we have the intermediate form
\begin{align*}
S(p^k) = p^{-2k} \sum_{l=0}^{k} \sum_{(a_1;a_2) = p^l} \prod_{i = 1}^s V(p^k,\lambda_i a_1, \lambda_i a_2).
\end{align*}
The change $a_i = p^l b_i$ and the corresponding change of $V(p^k,\lambda_i a_1, \lambda_i a_2)$ to
$p^{ls} V(p^{k-l},\lambda_i b_1, \lambda_i b_2)$ together with some elementary transformations
give us
\begin{align*}
\sum_{l=0}^{k} V(p^l) = p^{(2-s)k} S(p^k).
\end{align*}
Taking limits on both sides gives the statement of the lemma.
\end{proof}

The next task is to bound the number of solutions modulo $p^k$ from below.

\begin{lem} \label{plow-thm}
For each prime $p$ there is an $u = u(p) \in \N$, such that
\begin{align*}
S(p^k) \gg p^{(k-u)(s-2)}.
\end{align*}
\end{lem}
For the proof we need a version of Hensel's lemma rewritten in the language of $p$-adic valuations.
Denote by $\Z_p$ the $p$-adic integers and let $|\cdot|_p$ be the standard $p$-adic valuation.

\begin{lem} \label{Hensel-thm}
Suppose that $F_1(X_1,X_2), F_2(X_1,X_2) \in \Z_p[X_1,X_2]$, and that $a_0,b_0 \in \Z_p$ satisfy
\begin{align*}
\max\{|F_1(a_0,b_0)|_p,|F_2(a_0,b_0)|_p\} < |\Delta_0|_p^2,
\end{align*}
where $\Delta_0 = \Delta(F_1,F_2)|_{(a_0,b_0)} = 
\left( \frac{\partial F_1}{\partial X_1}\frac{\partial F_2}{\partial X_2} -
\frac{\partial F_2}{\partial X_1}\frac{\partial F_1}{\partial X_2}\right)_{(a_0,b_0)}$
is non-zero, $\frac{\partial F_1}{\partial X_1}$ etc. being formal derivatives.
Then there is a unique $(a,b) \in \Z_p \times \Z_p$, such that $F_1(a,b) = F_2(a,b) = 0$ and
\begin{align*}
\max\{|a-a_0|_p,|b-b_0|_p\} \leq p^{-1} \cdot |\Delta_0|_p.
\end{align*}
\end{lem}

\begin{proof}
This is Proposition 5.20 in \cite{Grb}.
\end{proof}

\begin{proof}[Proof of Lemma \rm\ref{plow-thm}]
The first step is to obtain a non-singular $p$-adic solution.
The system \eqref{eq-sys} can be written as a quadratic equation in $s-1$ variables
by inserting the linear equation into the quadratic one.
\begin{align*}
\lambda_s(\lambda_1x_1^2 + \ldots + \lambda_{s-1} x_{s-1}^2) + (\lambda_1x_1 + \ldots + \lambda_{s-1} x_{s-1})^2  = 0
\end{align*}
By a well known theorem of Meyer \cite{Meyer}, 
a quadratic form has a non-trivial (not all $x_i = 0$) $p$-adic solution,
as long as the number of variables is at least five.
We have $s-1 \geq 6$, so we have one variable `left'. We use this variable to ensure that our solution
is not only non-trivial, but also non-singular.
A solution of \eqref{eq-sys} is \emph{non-singular}, if one of the $2\times2$ matrices 
$\big(\begin{smallmatrix} 2\lambda_ix_i & 2\lambda_jx_j \\ \lambda_i & \lambda_j \end{smallmatrix}\big)$
has a non-zero determinant. This is the case if and only if $x_i \neq x_j$ for some $i \neq j$.
To achieve this, we fix $x_2 = 0$ and use Meyer's theorem with the remaining $s-2 \geq 5$ variables.
Since $x_2 = 0$ and $x_i \neq 0$ for some $i$ by non-triviality, the result follows.

Let $\yv \in \Z_p$ be the non-singular solution we have found above.
We switch back to the representation of \eqref{eq-sys} as two equations
and assume for simplicity that $y_1 \neq 0$, which may be achieved by renaming the variables.
Define for $\xv'' = (x_3,\ldots,x_s)$ the binary polynomials 
\begin{align*}
F_{i,\xv''}(x_1,x_2) := \lambda_1x_1^i + \lambda_2x_2^i + \lambda_3x_3^i + \ldots + \lambda_sx_s^i \qquad (i \in \{1,2\}).
\end{align*}
For the point $(a_0,b_0) = (y_1,0)$ we set $\Delta_0 = 2\lambda_1\lambda_2y_1$, as in Lemma \ref{Hensel-thm}.
For vectors $\xv''$ with $|x_j -y_j|_p < |\Delta_0|_p^2$ for $3 \leq j \leq s$, we have
\begin{align*}
|F_{i,\xv''}(y_1,0)|_p = |F_{i,\xv''}(y_1,0) - F_{i,\yv''}(y_1,0)|_p < |\Delta_0|_p^2
\end{align*}
by an elementary calculation, and, therefore, the condition of Lemma \ref{Hensel-thm} is satisfied.
This means that, as long as $\xv''$ is close enough to $\yv''$, we can extend $\xv''$ in a unique way
to a solution $\xv$.

Looking at this from the congruence point of view (by using the definition of the $p$-adic valuation), 
we have to ensure that $\xv''$ is congruent to $\yv''$
modulo $p^u$, where $u$ is defined by $|\Delta_0|_p^2 = p^{-u+1}$.
The number of choices for $\xv''$ modulo $p^k$, which are congruent to $\yv''$ modulo $p^u$ for $k \geq u$, is $(p^{k-u})^{s-2}$.
For each such choice $\xv''$ we have by Lemma \ref{Hensel-thm} a unique $(x_1,x_2) \in \Z_p$ with $F_{i,\xv''}(x_1,x_2) = 0$.
This leads to a unique solution modulo $p^k$ and the lemma follows.
\end{proof}

Now we get to the second main task, the lower bound for $J$.

\begin{lem} \label{J-est}
Under the usual assumptions for system \eqref{eq-sys}, we have
$J \gg N^{s-3}$,
where the implied constant depends on the coefficients $\lambda_i$.
\end{lem}

\begin{proof}
Similar to the p-adic case, we first show the existence of a non-singular real solution.
This is a solution with $x_i \neq x_j$ for some $i \neq j$, as observed in the proof of Lemma \ref{plow-thm}.
We will perform a pertubation argument and create a non-sigular solution from a singular one.
This is where condition (iii) for system \eqref{eq-sys} comes into play.
We know that there are at least two positive and two negative coefficients $\lambda_i$.
We sort and rename the coefficients and variables according to their signs and obtain
\begin{equation*}
\begin{split} 
\lambda_1 x_1^2 + \ldots + \lambda_t x_t^2 = \mu_1 x_{t+1}^2 + \ldots + \mu_t x_{t+r}^2, \\
\lambda_1 x_1 +  \ldots + \lambda_t x_t = \mu_1 x_{t+r} + \ldots + \mu_t x_{t+r},
\end{split}
\end{equation*}
where $r + t = s$ with $r,t \geq 2$ and $\lambda_i > 0,\mu_j > 0$.
By condition (i) for \eqref{eq-sys}, there is always the trivial solution $\xv = (1,\ldots,1)$.
Now fix all the variables $x_i$ to be $1$ except $x_1,x_2$ and $y_1 = x_{t+1}, y_2 = x_{t+2}$.
Then we are in the situation
\begin{equation} \label{xy-sys}
\begin{split} 
\lambda_1 x_1^2 + \lambda_2 x_2^2 + c = \mu_1 y_1^2 + \mu_2 y_{2}^2, \\
\lambda_1 x_1 + \lambda_2 x_2  + c  = \mu_1 y_1 + \mu_2 y_2,
\end{split}
\end{equation}
for some suitable constant $c \in \Z$.
Put $x_1 = 1 + \lambda_2 \theta, x_2 = 1 - \lambda_1 \theta$ and $y_1 = 1 + \mu_2 \varphi, y_2 = 1 - \mu_1 \varphi$.
System \eqref{xy-sys} is satisfied with $\theta = \varphi = 0$ and the linear equation for all choices of $\theta$ and $\varphi$.
This simplifies \eqref{xy-sys} to
\begin{align*}
\lambda_1\lambda_2(\lambda_1 + \lambda_2)\theta^2 = \mu_1\mu_2(\mu_1 + \mu_2)\varphi^2.
\end{align*}
Hence there is a solution with $\theta > 0 , \varphi > 0$, which is non-singular.

Now we prove the lower bound on $J$ by connecting it to the manifold of real solutions around
a non-singular point.
By \eqref{eq-J}, the dominated convergence theorem and \eqref{eq-tri}, we have
\begin{align*}
J = \lim_{P \to \infty} \int_{\R^2}  \tri(\alpha_2N^2/P) \tri(\alpha_1N/P)\prod_{i = 1}^s v(\lambda_i \alv) \, d\alv.
\end{align*}
Inserting the definition \eqref{eq-v} of $v$ and interchanging the order of integration
($\Lambda$ has finite support), leads to
\begin{align*}
J = \lim_{P \to \infty} \int_{[0,N]^s} \int_{\R^2}  
\tri(\alpha_2N^2/P) \tri(\alpha_1N/P) e(\alpha_2 K_2(\xv) + \alpha_1 K_1(\xv)) \,d\alv d\xv,
\end{align*}
where $K_i(\xv) = \sum_{l = 1}^s \lambda_lx_l^i$ (with coefficients $\lambda_l$ as in \eqref{eq-sys}).
The change of variables $x_i = Ny_i$ and $\beta_i = N^i\alpha_i$ separates the parameter $N$ and we obtain 
\begin{align*}
J = N^{s-3} \lim_{P \to \infty} \int_{[0,1]^s} \int_{\R^2}  \tri(\beta_2/P) \tri(\beta_1/P) e(\beta_2 K_2(\yv) + \beta_1 K_1(\yv)) \,d\bev d\yv.
\end{align*}
The two inner integrals are Fourier transforms and we get
\begin{align*}
J = N^{s-3} \lim_{P \to \infty} \int_{[0,1]^s} P \widehat{\tri}(-P K_2(\yv)) P\widehat{\tri}(-P K_1(\yv)) \, d\yv.
\end{align*}
The Fourier transform $\widehat{\tri}$ is given by $\frac{\sin(\pi y)^2}{(\pi y)^2}$ and is positive.
Therefore, we can estimate this integral from below by estimating $\frac{\sin(\pi y)^2}{(\pi y)^2}$ from below by $1/2$
when $|y| \leq 1/10$, for example.
We obtain the lower bound
\begin{align} \label{set}
4^{-1} P^2 \mu(\{\yv \in [0,1]^s: |K_2(\yv)| \leq (10P)^{-1}, |K_1(\yv)| \leq (10P^{-1}\}),
\end{align}
where $\mu$ is the Lebesgue measure.
By the argument above, the quadric has a non-singular real solution $\xv$ (let $x_1 \neq x_2$). By translation and
dilation invariance of the system \eqref{eq-sys}, we can assume that $\xv$ lies in $(0,1)^s$.\\
We use the implicit function theorem. This allows us to write 
$y_1 = \phi_1(\yv'')$ and $y_2 = \phi_2(\yv'')$ where $\yv'' = (y_3,\ldots,y_s)$ and $\phi_1, \phi_2$ are
differentiable functions, such that $K_i(\phi_1(\yv''), \phi_2(\yv''),\yv'') = 0$ ($i \in \{1,2\}$)
for all $\yv''$ in a neighbourhood of $\xv'' = (x_3,\ldots,x_s)$. 
If we fix one of these points $\yv''$, then we can vary the $y_1$ and $y_2$-coordinates by at least
$|\epsilon_i| = (40P|\lambda_i|)^{-1}$ ($i \in \{1,2\}$) without leaving the set in \eqref{set}. 
Therefore, its volume is $\gg P^{-2}$. 
\end{proof}

\appendix

\section{Collection of various estimates for exponential sums}   \label{AppA}

\begin{lem} \label{lin-est}
For the linear exponential sum $L_M$ in \eqref{lin-exp} with $M \geq 2$ we have
\begin{align*}
\int_{\T} |L_{M}(\alpha)| \,d\alpha \ll \log M.
\end{align*}
\end{lem}

\begin{proof}
The standard estimate $|L_{M}(\alpha)| \leq \min\{M, \|\alpha\|^{-1}\}$ 
follows by an application of the finite geometric sum formula.
Now decompose the integral into two parts $\|\alpha\| \leq 1/M$ and $\|\alpha\| > 1/M$. The result is immediate. 
\end{proof}

The next three lemmata are specialised versions of results from \cite{Vau}.
First we have an approximation to $V(\alv)$ by the local versions \eqref{eq-v} on the major arcs.

\begin{lem} \label{Vau-thm} 
Let $\alpha_i = b_i/q_i + \beta_i$ ($i=1,2$) with $b_i,q_i \in \Z$ and suppose that 
$q = \lcm(q_1,q_2)$ and $a_i = b_iqq_i^{-1}$. Then
\begin{align*}
V(\alv) = q^{-1} V(q,\av) v(\bev) + \Delta,
\end{align*}
where
\begin{align*}
\Delta \ll q (1 + |\beta_1|N + |\beta_2|N^2).
\end{align*}
\end{lem}

\begin{proof}
This is the special case $k=2$ of \cite[Theorem 7.2]{Vau}.
\end{proof}

Now we give an estimate for the local function $v(\alv)$.

\begin{lem} \label{v-est} 
We have the estimate
\begin{align*}
|v(\alv)| \ll N (1+|\alpha_1|N + |\alpha_2|N^2)^{-1/2}.
\end{align*}
\end{lem}

\begin{proof}
This is the special case $k=2$ of \cite[Theorem 7.3]{Vau}.
\end{proof}

The next result is an easy corollary of the previous lemma.

\begin{lem} \label{v-est2}
For $p > 4$ and $\T^2 = [-1/2,1/2]^2$ we have
\begin{align*}
\int_{\T^2} \psi^2(\alv)\, |v(\alv)|^p \,d\alv \ll_p N^{p-3},
\end{align*}
where $\psi(\alv)$ is defined in \eqref{eq-ind}.
\end{lem}

\begin{proof}
Use $(1+|\alpha_1|N + |\alpha_2|N^2)^{-2} \ll (1+|\alpha_1|N)^{-1}(1+|\alpha_2|N^2)^{-1}$ and $|\psi(\alv)| \leq 1$.
\end{proof}

The following result gives a well-known estimate for Gau\ss\ sums.
It allows us to bound $V(q,\av)$ and the quadratic 
exponential sum $G_{m_1}(q,a)$ appearing in Lemma \ref{transform}.

\begin{lem} \label{exp-est}
Let $F_2(\xv) = \xv^t M \xv$ be a quadratic form with $M \in \Z^{d \times d}$ a 
symmetric matrix of determinant $\Delta \neq 0$ and $F_{1}(\xv)$ an arbitrary affine linear polynomial. 
Let 
\begin{align*}
G(q,a) = \sum_{\xv \leq q} e_q(aF_2(\xv) + F_{1}(\xv)).
\end{align*}
Then for $(a;q) = 1$ we have
\begin{align*}
|G(q,a)| \ll_{d,\Delta} q^{d/2}
\end{align*}
uniformly in $F_1$. In particular, this implies
\begin{align*}
|V(q,\av)| \ll q^{1/2}.
\end{align*}
\end{lem}

\begin{proof}
Let $F'_{1}$ be the linear part of $F_{1}$.
By squaring $G(q,a)$, a change of variables and the $q$-periodicity of $e_q$ we obtain
\begin{align*}
|G(q,a)|^2 &= \sum_{\xv \leq q} \sum_{\yv \leq q} e_q(a (F_2(\xv)-F_2(\yv)) +  (F_{1}(\xv)-F_{1}(\yv))) \\
&= \sum_{\hv \leq q} \sum_{\yv \leq q} e_q(a (F_2(\yv+\hv)-F_2(\yv))) + (F_{1}(\yv + \hv)-F_{1}(\yv)) ) \\
&= \sum_{\hv \leq q} e_q(a F_2(\hv) + F'_{1}(\hv)) \sum_{\yv \leq q} e_q(a 2 \hv^t M \yv ) \\
&= q^d \sum_{\hv \leq q} e_q(a F_2(\hv) + F'_{1}(\hv)) \einsv_{q| 2 a M \hv}
\leq q^d \sum_{\hv \leq q} \einsv_{q| 2 a M \hv}.
\end{align*}
Since $(a;q) = 1$ the number of vectors $2aM\hv$ divisible by $q$ componentwise is bounded by $2^d \Delta$ 
and the estimate follows.
\end{proof}

We need a good estimate for $V(\alv)$ on the minor arcs as well.

\begin{lem} \label{ma-est} 
Assume that $|V(\alv)| \geq NQ^{-1/3}$ for $N \geq N_0(\epsilon)$ with $N^{3\epsilon} \leq Q$.
Then there exists $q \leq Q$ and $a_1,a_2 \in \Z$, such that $(q;a_1;a_2) = 1$ and
$|\alpha_iq - a_i| \leq QN^{-i}$ for $i \in \{1,2\}$.
\end{lem}

\begin{proof}
This is a special case of Theorem 5.1 in \cite{Baker}.
Choose the parameters in \cite{Baker} to be $M=1$, $P = NQ^{-1/3}$, $k = 2$ and 
$\epsilon > 0$ such that $N^{3\epsilon} \leq Q$.
\end{proof}

\section{From exponential correlation to a density increment} \label{sec-increment}

In this appendix we are going to prove the simple but crucial Lemma \ref{exp-prog}.
It enables us to transfer a correlation estimate for exponential sums into a
density increment on a subprogression.

Let $f$ be the balanced function as defined in \eqref{balanced}
and consider the corresponding exponential sum $V_f$ as defined in \eqref{eq-Sg}.

\begin{lem} \label{exp-prog}
If $|V_f(\alv)| \geq \eta N$ for some $\alv \in \T^2$ and $\eta > 0$, then there is
an arithmetic progression $P \subset \{1,2,\ldots,N\}$ of length $|P| \gg \eta^{2} N^{1/16}$ with
\begin{align*}
|\mathcal{A} \cap P| \geq (\delta + \eta/4) |P|.
\end{align*}
\end{lem}
Results of this kind are well known (see \cite{Gow}, for example), 
but in the current literature usually given in the $\Z/N\Z$-setting.
We give here a proof for completeness and
use the following classical result of Heilbronn in diophantine approximation.

\begin{lem} \label{diophant}
Let $\alpha \in \R$ and $Q \in \N$. Then there is $q \leq Q$ such that
$\|\alpha q^2\| \ll_{\epsilon} Q^{-1/2 + \epsilon}$.
\end{lem}

\begin{proof}
A proof can be found in \cite[p. 26]{Baker}.
\end{proof}

\begin{remark} 
Zaharescu \cite{Za} got the exponent $-4/7$ instead of $-1/2$. But for our purpose
the result of Heilbronn is sufficient and we even fix $\epsilon = 1/6$ for simplicity.
\end{remark}

\begin{proof}[Proof of Lemma \rm\ref{exp-prog}]
The idea is to cover $\{1,2,\ldots,N\}$ by long progressions in such a way that 
the function $h: n \mapsto e(\alpha_1 n^2 + \alpha_2 n)$ is almost constant on each progression.

By Lemma \ref{diophant} we can find $q \leq N^{3/4}$, such that $\|\alpha_2 q^2\| \leq C N^{-1/4}$
for some absolute constant $C \geq 1$.
We can assume that $N \geq (2^8C/\eta^2)^{16}$. 
Otherwise, the lemma is true for a singleton $P = \{n\}$ with $n \in \mathcal{A}$.
We decompose $\{1,2, \ldots, N\}$ in congruence classes modulo $q$ and split those 
into progressions $\{m + kq: 0 \leq k < K\}$ of length $K$ between $2^{-4}C^{-1}\eta N^{1/8}$ and $2^{-3}C^{-1}\eta N^{1/8}$.
The expression $\alpha_2 n^2 + \alpha_1 n$ varies on such a progression by at most
\begin{align*}
& \|(\alpha_2 (m+qk)^2 + \alpha_1 (m+qk)) - (\alpha_2 (m+ql)^2 + \alpha_1 (m+ql))\| \\
 =\ & \|(2\alpha_2 mq + \alpha_1 q)(k-l) + \alpha_2 q^2(k^2-l^2)\| \\
 \leq\ & \|(2\alpha_2 mq + \alpha_1 q)(k-l)\| + \|\alpha_2 q^2\|K^2 \leq\ \|\beta (k-l)\| + 2^{-6}\eta^2,
\end{align*}
where $\beta := 2\alpha_2 mq + \alpha_1 q$. For a given progression from the decomposition above,
the expression $\beta$ is constant and we have reduced our quadratic problem to a linear one.
By Dirichlet's approximation theorem we can find a value $r \leq K^{1/2}$ with $\|\beta r\| \leq K^{-1/2}$.
We partition the given progression into congruence classes modulo $r$ and then further in 
subprogressions $P_i$ ($i \in I$, where $I$ is a set of indices) of length between $2^{-6}\eta K^{1/2}$ and $2^{-5} \eta K^{1/2}$. 
This leads to
\begin{align*}
\|\beta (k-l)\| = \|\beta tr\| \leq 2^{-5}\eta.
\end{align*}

Observe that for $\|x_1-x_2\| \leq y$ we get $|e(x_1)-e(x_2)| \leq 2\pi y$.
Therefore, the function $h(n) = e(\alpha_1 n^2 + \alpha_2 n)$ varies at most $\eta/2$ on each of these progressions.
Now write $h(n) = h(m_i) + (h(n)-h(m_i))$ on
each progression $P_i$, where $m_i \in P_i$. 
Since $|h(n)-h(m_i)|$ is bounded by $\eta/2$, we have
\begin{align*}
\eta N &\leq |V_f(\alv)| = \Big|\sum_{n \leq N} f(n) h(n) \Big|
= \Big|\sum_{i \in I} \sum_{n \in P_i} f(n) h(n) \Big|\\
&\leq \Big|\sum_{i \in I} h(m_i) \sum_{n \in P_i} f(n) \Big| + \eta N/2
\leq \sum_{i \in I}  \Big| \sum_{n \in P_i} f(n) \Big| + \eta N/2.
\end{align*}
The average $\sum_{i \in I} \sum_{n \in P_i} f(n)$ is zero by definition of $f$.
We can add it to the right hand side and use the identity $|x| + x = 2 \max\{0,x\}$. This gives
\begin{align*}
\sum_{i \in I}  \max\Big\{0, \sum_{n \in P_i} f(n) \Big\} \geq \eta N/4.
\end{align*}
It follows that there has to be at least one $i \in I$ with $\sum_{n \in P_i} f(n) \geq \eta |P_i|/4$.
The result follows from the rearranging and using the formula $f(n) = 1_{\mathcal{A}}(n) - \delta$.
\end{proof}

\end{document}